\documentclass[submission,copyright,creativecommons]{eptcs}
\usepackage{graphicx}
\usepackage{amsthm}
\usepackage{thmtools, thm-restate} 
\usepackage{amsmath}
\usepackage{hyperref}
\usepackage[noabbrev]{cleveref}
\graphicspath{ {./images/} }
\usepackage[utf8]{inputenc}
\usepackage{amsmath}
\usepackage{mathtools}
\usepackage{dsfont}
\usepackage{amssymb}
\usepackage{amssymb}
\usepackage{graphicx}
\usepackage{blindtext}
\usepackage{quiver}
\usepackage{comment}
\usepackage{subcaption}
\usepackage{tikz-cd}
\usepackage{todonotes}
\usepackage{amsmath,amssymb,amsthm, amsfonts, xypic, graphicx}
\usepackage{setspace}
\newtheorem{theorem}{Theorem}[section]

\newtheorem{example}[theorem]{Example}
\newtheorem{definition}[theorem]{Definition}
\newtheorem{lemma}[theorem]{Lemma}

\newtheorem{notation}[theorem]{Notation}

\usepackage{tikz}
\usepackage{tikz-cd}
\usepackage{amsmath}
\usepackage{tikz}
\usetikzlibrary{decorations.markings,positioning}
\usepackage{array}
\usepackage{enumerate}
\usepackage{url}
\usetikzlibrary{automata, positioning, arrows}

\newcommand{\Set}{\mathsf{Set}}
\newcommand{\Finset}{\mathsf{Finset}}

\newcommand{\dDWD}{\mathsf{dDWD}}
\newcommand{\DWD}{\mathsf{DWD}}
\newcommand{\Xin}{{X_\textrm{in}}}
\newcommand{\Xout}{{X_\textrm{out}}}
\newcommand{\Yin}{{Y_\textrm{in}}}
\newcommand{\Yout}{{Y_\textrm{out}}}
\newcommand{\Zin}{{Z_\textrm{in}}}
\newcommand{\Zout}{{Z_\textrm{out}}}
\newcommand{\binomx}{\binom{\Xin}{\Xout}}
\newcommand{\binomy}{\binom{\Yin}{\Yout}}
\newcommand{\binomz}{\binom{\Zin}{\Zout}}
\newcommand{\W}{W}
\newcommand{\Win}{{W_\textrm{in}}}
\newcommand{\Wout}{{W_\textrm{out}}}
\newcommand{\R}{\mathbb R}
\newcommand{\Poset}{\mathsf{Poset}}

\newcommand{\xin}{{x_\textrm{in}}}
\newcommand{\xout}{{x_\textrm{out}}}
\newcommand{\yin}{{y_\textrm{in}}}

\newcommand{\oxin}{\overline{x}_\textrm{in}}
\newcommand{\oxout}{\overline{x}_\textrm{out}}

\newcommand{\ophi}{\overline{\phi}}
\newcommand{\ophiout}{\overline{\phi}_\textrm{out}}
\newcommand{\ophiin}{\overline{\phi}_\textrm{in}}
\newcommand{\ovarphi}{\overline{\varphi}}

\newcommand{\inflow}{\textrm{inflow}}
\newcommand{\outflow}{\textrm{outflow}}
\newcommand{\stock}{\textrm{stock}}
\newcommand{\flow}{\textrm{flow}}
\newcommand{\inport}{\textrm{in port}}
\newcommand{\outport}{\textrm{out port}}

\newcommand{\sumvar}{\text{sum var}}
\newcommand{\sumlink}{\text{sum link}}
\newcommand{\aux}{\text{var}}
\newcommand{\variable}{\text{var}}
\newcommand{\outlink}{\text{out link}}
\newcommand{\inlink}{\text{in link}}
\newcommand{\stocklink}{\text{stock link}}
\newcommand{\varlink}{\text{var link}}
\newcommand{\stocksumlink}{\text{stock sum link}}

\newcommand{\Rst}{\R^{F(\stock)}}
\newcommand{\Rsv}{\R^{F(\sumvar)}}
\newcommand{\Rv}{\R^{F(\aux)}}

\newcommand{\SF}{\mathsf{SF}}
\newcommand{\Hf}{\mathsf{H_f}}

\newcommand{\Mealy}{\mathsf{Mealy}}
\newcommand{\Dynam}{\Mealy}

\newcommand{\dep}{R}
\renewcommand{\xto}{\xrightarrow}
\newcommand{\xfrom}{\xleftarrow}

\usepackage{comment}

\newcommand{\Path}{\mathsf{Path}}
\DeclareMathOperator{\ob}{\mathsf{ob}}

\newcommand{\defined}[1]{\textbf{#1}}

\usepackage{rotating}

\ifpdf
  \usepackage{underscore}         
  \usepackage[T1]{fontenc}        
\else
  \usepackage{breakurl}           
\fi

\title{Dependent Directed Wiring Diagrams for Composing Instantaneous Systems}
\author{Keri D'Angelo
\institute{Cornell University\\ Ithaca, NY, USA}
\email{keri@cs.cornell.edu}
\and
Sophie Libkind
\institute{Topos Institute\\
Berkeley, CA, USA}
\email{sophie@topos.institute}
}
\def\titlerunning{Composing instantaneous systems}
\def\authorrunning{K. D'Angelo and S. Libkind}

\usepackage{amsthm}
\usepackage{stmaryrd}
\usepackage{newpxtext}
\usepackage[varg,bigdelims]{newpxmath}
\usepackage{tikz}

\usetikzlibrary{
	cd,
	math,
	backgrounds,
	decorations.markings,
	decorations.pathreplacing,
	positioning,
	arrows.meta,
	circuits.logic.US,
	shapes,
	calc,
	fit,
	trees,
	quotes}

\tikzset{
  WD/.style={
  	label/.style={
    	font=\everymath\expandafter{\the\everymath\scriptstyle},
      inner sep=0pt,
      node distance=2pt and -2pt},
  	label distance=-2pt,
  	every to/.style={draw},
    semithick,
    node distance=\bbx and \bby,
    decoration={markings, mark=at position \stringdecpos with \stringdec},
    bb port length=3pt,
  	bb port sep=1,
		bb inside color=white,
		bb outside color=black,
	 	bbx = .4cm,
		bb min width=.4cm,
	  bby = 2ex,
	  bb penetrate=0,
	  bb rounded corners=2pt,
	  dot size=3pt,
    shell size = 16pt,
   	penetration = 0pt,
    link size = 2pt,
    shell color = blue,
  	shell inside color=\pcolor!20,
 	  shell outside color=\pcolor!50!black,
  	surround sep=2pt,
    ar/.style={postaction={decorate}},
  	execute at begin picture={\tikzset{
  		x=\bbx, y=\bby, 
			circuit logic US, tiny circuit symbols
			}
		}
  },
  beamer/.style={
  	bbx=.4cm,
		bb min width=.4cm,
		bby=6pt,
		bb port length=3pt,
		bb port sep=.5,
  	dot size=1pt,
    shell size = 11pt, 
   	penetration = 0pt,
    link size = 1pt,
    shell color = blue,
    surround sep=1pt,
    inner sep=1pt,
    font=\tiny,
    bb inside color=\picolor,
    bb outside color=\pocolor,
	},
	bb standard colors/.style={bb inside color=white, bb outside color=black},
	bb inside color/.store in=\bbicolor,
	bb outside color/.store in=\bbocolor,
  bbx/.store in=\bbx,
  bby/.store in=\bby,
  bb port sep/.store in=\bbportsep,
  bb port length/.store in=\bbportlen,
  bb penetrate/.store in=\bbpenetrate,
  bb min width/.store in=\bbminwidth,
  bb rounded corners/.store in=\bbcorners,
  bb/.code 2 args={
    \pgfmathsetlengthmacro{\bbheight}{\bbportsep * (max(#1,#2)+1) * \bby}
    \pgfkeysalso{
      draw=\bbocolor,
      fill=\bbicolor,
      minimum height=\bbheight,
      minimum width=\bbminwidth,
      outer sep=0pt,
      rounded corners=\bbcorners,
      thick,
      prefix after command={\pgfextra{\let\fixname\tikzlastnode}},
      append after command={\pgfextra{\draw
      	\ifnum #1=0{} \else foreach \i in {1,...,#1} {
        	($(\fixname.north west)!{(2*\i-1)/(2*#1)}!(\fixname.south west)$) +(-\bbportlen,0) coordinate (\fixname_in\i) -- +(\bbpenetrate,0) coordinate (\fixname_in\i')}\fi 
        \ifnum #2=0{} \else foreach \i in {1,...,#2} {
        	($(\fixname.north east)!{(2*\i-1)/(2*#2)}!(\fixname.south east)$) +(-
\bbpenetrate,0) coordinate (\fixname_out\i') -- +(\bbportlen,0) coordinate (\fixname_out\i)}\fi
;
       }}}
		},
	dot size/.store in=\dotsize,
	dot/.style={
		circle, draw, thick, inner sep=0, fill=black, minimum width=\dotsize
	},
	bb name/.style={
    append after command={
		\pgfextra{\node[anchor=north] at (\fixname.north) {#1};}
		}
	},
  shell size/.store in=\psize,
	penetration/.store in=\penetration,
  spacing/.store in=\spacing,
  link size/.store in=\lsize,
  shell color/.store in=\pcolor,
 	shell inside color/.store in=\picolor,
 	shell outside color/.store in=\pocolor,
 	surround sep/.store in=\ssep,
 	link/.style={
  	circle, 
  	draw=black, 
  	fill=black,
  	inner sep=0pt, 
 		minimum size=\lsize
 	},
  shell/.style={
 		circle, 
 		draw = \pocolor, 
  	fill = \picolor,
  	minimum size = \psize
  },
  func/.style={
  	shell,
		rectangle,
		rounded corners=.5*\psize,
		inner ysep=.125*\psize,
		minimum width=1.125*\psize,
		inner xsep=.25*\psize,
  },
  funcr/.style={
    func,
    rectangle round north west=false, 
		rectangle round south west=false,
  },
  funcl/.style={
    func,
		rectangle round north east=false, 
		rectangle round south east=false,
  },
  funcu/.style={
    func,
		rectangle round south east=false, 
		rectangle round south west=false,
  },
  funcd/.style={
    func,
		rectangle round north east=false, 
		rectangle round north west=false,
  },
  outer shell/.style={
 		ellipse, 
 		draw,
  	inner sep=\ssep,
  	color=gray,
 	},
  intermediate shell/.style={
 		ellipse,
 		dashed, 
  	draw,
  	inner sep=\ssep,
 		color=\pocolor,
 	},
 }

\tikzset{
	oriented WD/.style={
		every to/.style={out=0,in=180,draw},
    label/.style={
    	font=\everymath\expandafter{\the\everymath\scriptstyle},
      inner sep=0pt,
      node distance=2pt and -2pt},
    semithick,
    node distance=1 and 1,
    decoration={markings, mark=at position \stringdecpos with \stringdec},
    ar/.style={postaction={decorate}},
    execute at begin picture={\tikzset{
    	x=\bbx, y=\bby,
      every fit/.style={inner xsep=\bbx, inner ysep=\bby}}}
    },
    string decoration/.store in=\stringdec,
    string decoration={\arrow{stealth};},
    string decoration pos/.store in=\stringdecpos,
    string decoration pos=.7,
    bbx/.store in=\bbx,
    bbx = 1.5cm,
    bby/.store in=\bby,
    bby = 1.ex,
    bb port sep/.store in=\bbportsep,
    bb port sep=1.5,
    bb port length/.store in=\bbportlen,
    bb port length=4pt,
    bb penetrate/.store in=\bbpenetrate,
    bb penetrate=0,
    bb min width/.store in=\bbminwidth,
    bb min width=1cm,
    bb min height/.store in=\bbminheight,
    bb min height=1cm,
    bb rounded corners/.store in=\bbcorners,
    bb rounded corners=2pt,
    bb spider/.style={
    	bb port sep=1, bb port length=10pt, bbx=.4cm, bb min width=.4cm, bby=.8ex},
    bb small/.style={
    	bb port sep=1, bb port length=2.5pt, bbx=.4cm, bb min width=.4cm, bby=.7ex},
		bb medium/.style={
			bb port sep=1, bb port length=2.5pt, bbx=.4cm, bb min width=.4cm, bby=.9ex},
    bb/.code n args={4}{
    	\pgfmathsetlengthmacro{\bbheight}{\bbportsep * (max(#1,#2)+1) * \bby}
    	\pgfmathsetlengthmacro{\bbwidth}{\bbportsep * (max(#3,#4)+1) * \bby}
      \pgfkeysalso{draw,minimum height=\bbminheight,minimum
       width=\bbminwidth,outer sep=0pt,
         rounded corners=\bbcorners,thick,
         prefix after command={\pgfextra{\let\fixname\tikzlastnode}},
         append after command={\pgfextra{\draw
            \ifnum #1=0{} \else foreach \i in {1,...,#1} {
            	($(\fixname.north west)!{\i/(#1+1)}!(\fixname.south west)$) +(-\bbportlen,0) coordinate (\fixname_in\i) -- +(\bbpenetrate, 0) coordinate (\fixname_in\i')}\fi 
            \ifnum #2=0{} \else foreach \i in {1,...,#2} {
            	($(\fixname.north east)!{\i/(#2+1)}!(\fixname.south east)$) +(-
\bbpenetrate,0) coordinate (\fixname_out\i') -- +(\bbportlen,0) coordinate (\fixname_out\i)}\fi

\ifnum #3=0{} \else foreach \i in {1,...,#3} {
            	($(\fixname.north west)!{\i/(#3+1)}!(\fixname.north east)$) +(0, \bbportlen) coordinate (\fixname_top\i) -- +(0,-\bbpenetrate) coordinate (\fixname_top\i')}\fi 
\ifnum #4=0{} \else foreach \i in {1,...,#4} {
            	($(\fixname.south west)!{\i/(#4+1)}!(\fixname.south east)$) +(0, \bbpenetrate) coordinate (\fixname_bot\i) -- +(0,-\bbportlen) coordinate (\fixname_bot\i')}\fi 
;
           }}}
		},
			bb name/.style={
     	append after command={
				\pgfextra{\node[anchor=north] at (\fixname.north) {#1};}
			}
		}
  }

  \tikzset{
  	unoriented WD/.style={
  		every to/.style={draw},
  		shorten <=-\penetration, shorten >=-\penetration,
  		label distance=-2pt,
  		thick,
  		node distance=\spacing,
  		execute at begin picture={\tikzset{
  			x=\spacing, y=\spacing}}
  		},
  	pack size/.store in=\psize,
  	pack size = 8pt,
  	spacing/.store in=\spacing,
  	spacing = 8pt,
  	link size/.store in=\lsize,
  	link size = 2pt,
		penetration/.store in=\penetration,
		penetration = 2pt,
  	pack color/.store in=\pcolor,
  	pack color = blue,
  	pack inside color/.store in=\picolor,
  	pack inside color=blue!20,
  	pack outside color/.store in=\pocolor,
  	pack outside color=blue!50!black,
  	surround sep/.store in=\ssep,
  	surround sep=8pt,
  	link/.style={
  		circle, 
  		draw=black, 
  		fill=black,
  		inner sep=0pt, 
  		minimum size=\lsize
  	},
  	pack/.style={
  		circle, 
  		draw = \pocolor, 
  		fill = \picolor,
  		inner sep = .25*\psize,
  		minimum size = \psize
  	},
  	outer pack/.style={
  		ellipse, 
  		draw,
  		inner sep=\ssep,
  		color=\pocolor,
  	},
  	intermediate pack/.style={
  		ellipse,
  		dashed, 
  		draw,
  		inner sep=\ssep,
  		color=\pocolor,
  	},
  }

\tikzset{
	spider diagram/.style={
		every to/.style={out=0, in=180, draw, thick},
		thick
	},
	dot size/.store in=\dotsize,
	dot size = 5pt,
	dot fill/.store in=\dotfill,
	dot fill = black,
	leg length/.store in=\leglen,
	leg length = 15pt,
	baby/.style={dot size = 2pt, leg length = 6pt},
	young/.style={dot size = 3pt, leg length = 10pt},
	special spider/.code n args={4}{
		\pgfkeysalso{circle, draw, thick, inner sep=0, fill=\dotfill, minimum width=\dotsize,
  		prefix after command={\pgfextra{\let\fixname\tikzlastnode}},
  		append after command={\pgfextra{
  			\ifnum #1=0{} \else {\foreach \i in {1,...,#1} {
					\tikzmath{\anglei={-90*(#1+1-2*\i)/#1};}
  				\draw [thick]
						(\fixname) .. controls 
						($(\fixname.center)-(\anglei:#3/3)$) and ($(\fixname.center)-(\anglei:#3*2/3)$) .. 
						({$(\fixname)-(\anglei:#3*2/3)$}-|{$(\fixname)-(#3,0)$}) coordinate (\fixname_in\i);
  			}}\fi
  			\ifnum #2=0{} \else {\foreach \i in {1,...,#2} {
					\tikzmath{\anglei={90*(#2+1-2*\i)/#2};}
  				\draw [thick]
						(\fixname.center) .. controls 
						($(\fixname.center)+(\anglei:#4/3)$) and ($(\fixname.center)+(\anglei:#4*2/3)$) .. 
						({$(\fixname.center)+(\anglei:#4*2/3)$}-|{$(\fixname.center)+(#4,0)$}) coordinate (\fixname_out\i);
  			}}\fi
  		}}
		}
	},
	spider/.code 2 args={
		\pgfkeysalso{special spider={#1}{#2}{\leglen}{\leglen}}
	}
}

\tikzset{Yonepart/.pic={
	\node[bb={1}{2},bb name = {\tiny$X_{11}$}] (X11) {};
	\node[bb={2}{2},below right=of X11,bb name = {\tiny$X_{12}$}] (X12) {};
	\node[bb={2}{1}, above right=of X12,bb name = {\tiny$X_{13}$}] (X13) {};
	\node[bb={2}{2}, fit={($(X11.north west)+(.3,1.5)$) (X12)  ($(X13.east)+(-.3,0)$)},bb name = {\scriptsize $Y_1$}] (Y1) {};
	\draw (Y1_in1') to (X11_in1);	
	\draw (Y1_in2') to (X12_in2);
	\draw (X11_out1) to (X13_in1);
	\draw (X11_out2) to (X12_in1);
	\draw (X12_out1) to (X13_in2);
	\draw (X12_out2) to (Y1_out2');
	\draw (X13_out1) to (Y1_out1');
	\coordinate (bottombox) at ($(X12.south)$);
	\coordinate (rightbox) at ($(X13.east)$);
	\coordinate (Y1northwest) at ($(Y1.north west)$);
	}
}

\tikzset{Ytwopart/.pic={
	\node[bb={2}{2}, bb name = {\tiny$X_{21}$}] (X21) {};
	\node[bb={1}{2},above right=-1 and 1 of X21,bb name = {\tiny$X_{22}$}] (X22) {};
	\node[bb={1}{2}, fit={($(X21.south west)+(-.25,0)$) ($(X22.north east)+(.25,3.5)$)},bb name = {\scriptsize$Y_2$}] (Y2){};
	\draw (Y2_in1') to (X21_in2);
	\draw (X21_out1) to (X22_in1);
	\draw (X22_out2) to (Y2_out1');
	\draw let \p1=(X22.south east), \p2=($(Y2_out2)$), \n1={\y1-\bby}, \n2=\bbportlen in
	  (X21_out2) to (\x1+\n2,\n1) -- (\x1+\n2,\n1) to (Y2_out2');
	\draw let \p1=(X22.north east), \p2=(X21.north west), \n1={\y1+\bby}, \n2=\bbportlen in
          (X22_out1) to[in=0] (\x1+\n2,\n1) -- (\x2-\n2,\n1) to[out=180] (X21_in1);
          }
}

\tikzset{SmallNeuronPic/.pic={
 \node[bb={3}{1}] (N1) {$\scriptstyle N_1$};
  \node[bb={2}{1}, above right=.7 and 3.5 of N1] (N2) {$\scriptstyle N_2$};
  \node[bb={2}{1}, below =of N2] (N3) {$\scriptstyle N_3$};
  \node[bb={3}{1}, below =of N3] (N4) {$\scriptstyle N_4$};
  \node[bb={6}{8}, fit={($(N1.west)-(.5,0)$) ($(N2.north)+(0,2)$) ($(N3.east)+(1.5,0)$) ($(N4.south)-(0,1)$)}, bb name={$\scriptstyle X$}] (X) {};
  \draw (X_in1') to (N2_in1);
  \draw (X_in2') to (N1_in1);
  \draw (X_in3') to (N1_in2);
  \draw (X_in4') to (N1_in3);
  \draw (X_in6') to (N4_in2);
  \draw (N1_out1) to (N2_in2);
  \draw (N1_out1) to (N3_in1);
  \draw (N1_out1) to (N4_in1);
  \draw (N2_out1) to (X_out1');
  \draw (N2_out1) to (X_out2');
  \draw (N2_out1) to (X_out3');
  \draw (N3_out1) to (X_out4');
  \draw (N3_out1) to (X_out5');
  \draw (N3_out1) to (X_out6');
  \draw (N4_out1) to (X_out7');
  \draw (N4_out1) to (X_out8'); 
  \draw (X_in5') to[looseness=2] (N3_in2);
  \draw let \p1=(N4.south east), \p2=(N4.south west), \n1={\y2-\bby}, \n2=\bbportlen in
          (N3_out1) to[in=0] (\x1+\n2,\n1) -- (\x2-\n2,\n1) to[out=180] (N4_in3);
}
}

\tikzset{SmallNeuronDashed/.pic={
 \node[bb={3}{1}] (N1) {$\scriptstyle N_1$};
  \node[bb={2}{1}, above right=.7 and 3.5 of N1] (N2) {$\scriptstyle N_2$};
  \node[bb={2}{1}, below =of N2] (N3) {$\scriptstyle N_3$};
  \node[bb={3}{1}, below =of N3] (N4) {$\scriptstyle N_4$};
  \node[bb={6}{8}, fit={($(N1.west)-(.5,0)$) ($(N2.north)+(0,2)$) ($(N3.east)+(1.5,0)$) ($(N4.south)-(0,1)$)}, bb name={$\scriptstyle X$}] (X) {};
  \draw[dashed] (X_in1') to (N2_in1);
  \draw[dashed] (X_in2') to (N1_in1);
  \draw[dashed] (X_in3') to (N1_in2);
  \draw[dashed] (X_in4') to (N1_in3);
  \draw[dashed] (X_in6') to (N4_in2);
  \draw[dashed] (N1_out1) to (N2_in2);
  \draw[dashed] (N1_out1) to (N3_in1);
  \draw[dashed] (N1_out1) to (N4_in1);
  \draw[dashed] (N2_out1) to (X_out1');
  \draw[dashed] (N2_out1) to (X_out2');
  \draw[dashed] (N2_out1) to (X_out3');
  \draw[dashed] (N3_out1) to (X_out4');
  \draw[dashed] (N3_out1) to (X_out5');
  \draw[dashed] (N3_out1) to (X_out6');
  \draw[dashed] (N4_out1) to (X_out7');
  \draw[dashed] (N4_out1) to (X_out8'); 
  \draw[dashed] (X_in5') to[looseness=2] (N3_in2);
  \draw[dashed] let \p1=(N4.south east), \p2=(N4.south west), \n1={\y2-\bby}, \n2=\bbportlen in
          (N3_out1) to[in=0] (\x1+\n2,\n1) -- (\x2-\n2,\n1) to[out=180] (N4_in3);
}
}

\tikzset{SmallNestingPic/.pic={
\path (0,0) pic [purple] {Yonepart};
\path ($(rightbox)+(5,-5)$) pic [orange] {Ytwopart};
 
\node[bb={1}{2}, fit={($(Y1northwest)+(-.5,4)$) ($(Y2.south east)+(1,0)$)}, bb name={\small $Z$}] (Z) {};
\draw (Z_in1') to (Y1_in2);
\draw let \p1=(Y2.north west),\p2=(Y2.north east),\n1={\y2+\bby},\n2=\bbportlen in
  (Y1_out1) to (\x1+\n2,\n1)--(\x2+\n2,\n1) to (Z_out1');
\draw (Y1_out2) to (Y2_in1);
\draw (Y2_out2) to (Z_out2');
\draw let \p1=(Y2.north east), \p2=(Y1.north west), \n1={\y2+\bby}, \n2=\bbportlen in
          (Y2_out1) to[in=0] (\x1+\n2,\n1) -- (\x2-\n2,\n1) to[out=180] (Y1_in1);
          }
}

\tikzset{Zredgreen/.pic={
\node[bb={2}{2}, green!50!black, bb name = $\scriptstyle Y_1$] (YY1) {};
\node[bb={1}{2}, red, below right=-1 and 2 of YY1, bb name=$\scriptstyle Y_2$] (YY2) {};
\node[bb={1}{2}, fit={($(YY1.north west)+(-.5,4)$) ($(YY2.south east)+(.5,-2)$)}, bb name={\scriptsize $Z$}] (Z) {};
\draw (Z_in1') to (YY1_in2);
\draw (YY1_out1) to (Z_out1');
\draw (YY1_out2) to (YY2_in1);
\draw (YY2_out2) to (Z_out2');
\draw let \p1=(YY2.north east), \p2=(YY1.north west), \n1={\y2+\bby}, \n2=\bbportlen in
          (YY2_out1) to[in=0] (\x1+\n2,\n1) -- (\x2-\n2,\n1) to[out=180] (YY1_in1);
}
}

\tikzset{Zcombined/.pic={
	\node[bb={1}{2},green!25!black,bb name = {\tiny$X_{11}$}] (X11) {};
	\node[bb={2}{2},green!25!black,below right=of X11,bb name = {\tiny$X_{12}$}] (X12) {};
	\node[bb={2}{1}, green!25!black,above right=of X12,bb name = {\tiny$X_{13}$}] (X13) {};
	\draw (X11_out1) to (X13_in1);
	\draw (X11_out2) to (X12_in1);
	\draw (X12_out1) to (X13_in2);

	\node[bb={2}{2}, red!30!black, below right = 0 and 1.25 of X12, bb name = {\tiny$X_{21}$}] (X21) {};
	\node[bb={1}{2}, red!30!black, above right=-1 and 1 of X21,bb name = {\tiny$X_{22}$}] (X22) {};
	\draw (X21_out1) to (X22_in1);
	\draw let \p1=(X22.north east), \p2=(X21.north west), \n1={\y1+\bby}, \n2=\bbportlen in
          (X22_out1) to[in=0] (\x1+\n2,\n1) -- (\x2-\n2,\n1) to[out=180] (X21_in1);
        
        \node[bb={1}{2}, fit = {($(X11.north east)+(-1,3)$) (X12) (X13) ($(X21.south)+(0,-1)$) ($(X22.east)+(.5,0)$)}, bb name ={\scriptsize $Z$}] (Z) {};
	
	\draw (Z_in1') to (X12_in2);
	\draw (X13_out1) to (Z_out1');
	\draw (X12_out2) to (X21_in2);
	\draw let \p1=(X22.south east),\n1={\y1-\bby}, \n2=\bbportlen in
	  (X21_out2) to (\x1+\n2,\n1) to (Z_out2');
	\draw let \p1=(X22.north east), \p2=(X11.north west), \n1={\y2+\bby}, \n2=\bbportlen in
          (X22_out2) to[in=0] (\x1+\n2,\n1) -- (\x2-\n2,\n1) to[out=180] (X11_in1);
}
}
\begin{document}
\maketitle

\begin{abstract}
Directed wiring diagrams can be used as a composition pattern for composing input/output systems such as Moore machines. In a Moore machine, the input parametrizes an internal state and the internal state defines the output. 
    Because the value of the output is shielded from the input by the internal state, Moore machines can compose by connecting the output of any machine to the input of any other machine. These connections are defined by the trace wires in a directed wiring diagram. 
    Unlike Moore machines, Mealy machines allow the output to be directly and instantaneously affected by the input. In order to compose such machines via directed wiring diagrams, it is necessary to avoid cycles between trace wires in the wiring digram and dependencies of outputs on inputs. To capture these patterns of composition, we introduce an operad of \textit{dependent directed wiring diagrams}. We then define an algebra of Mealy machines on this operad and an algebra of stock and flow diagrams in which the values of auxiliary variables are parameterized by inputs. Finally, we give a semantics for this algebra of stock and flow diagrams by giving a morphism of algebras from stock and flow diagrams into Mealy machines. 
\end{abstract}


In a system with many interacting components, wiring diagrams can provide a visual yet formal representation of how the components are composed. In particular, in \textit{directed wiring diagrams}, information flows from specified sources to designated targets. These diagrams define a composition pattern in which systems influence each other's behavior while maintaining their own independent dynamics. Directed wiring diagrams give a composition pattern for composing Moore machines as well as other systems such as parameterized ordinary differential equations (ODEs), automata, and hybrid dynamical systems~\cite{lerman2020hybrid, DWD,libkind2020rsm, Moore}. Wiring diagrams are a special case of lens-based composition~\cite{jazmyers2021double, generalizelens}.

A real-valued Moore machine consists of (1) a state space $\R^S$, an input space $\R^{\Xin}$, and an output space $\R^\Xout$, (2) an update function $\R^\Xin \times \R^S \to \R^S$  and (3) a readout function $\R^S \to \R^\Xout$. To understand how directed wiring diagrams give a composition pattern for composing Moore machines, consider the following example.


\begin{example}\label{example:moore-composite}

Take two Moore machines both with $S = \Xin = \Xout = 1$ and with the following updates and readouts:
\begin{align*}
    u_1(a,x) &= a + x & u_2(b,y) &= b\\
    r_1(x) &= x & r_2(y) &= y
\end{align*}
The first Moore machine updates the state by adding the input to the state while the second Moore machine updates the state by replacing it with the input. Both Moore machines output the state directly. We can \\

\pagebreak

compose these Moore machines according to the wiring diagram below. 

\begin{center}
    \begin{tikzpicture}[oriented WD, bbx = .5cm, bby =0.25cm, bb port length=4pt, bb port sep=.75, bb min width=1cm, bb min height=1cm]
    \node[bb={0}{0}{1}{1}] (A){$x = a + x$};
    \node[bb={0}{0}{1}{1}, right= of A] (B){$y = b$};
    \node[bb={0}{0}{0}{0}, fit={($(A.north west)+(0,1.25)$)($(B.south east)+(0, -3.5)$)}] (tot){};



    \draw[ar, color=violet] let \p1=(A.south west), \p2=(A.north west), \n1=\bbportlen, \n2=\bby in
     (B_bot1') to [out=-90, in=-90] (\x1 - \bby, \y1 - \n1) to [out=90, in=-90] (\x2 - \bby, \y2 +\n1) to [out=90, in=90] (A_top1);

    \draw[ar, color=violet] 
     let \p1=(A.south east), \p2=(B.north west), \n1=\bbportlen, \n2=\bby in
     (A_bot1') to [out=-90, in=-90] (\x1 + \bby, \y1 - \n1) to [out=90, in=-90] (\x2 - \bby, \y2 +\n1) to [out=90, in=90] (B_top1);

     \draw[label]
node [right = 2pt of A_top1] {$a$}
node [left = 2pt of A_bot1, yshift = -1ex] {$x$}
;

\draw[label]
node [right = 2pt of B_top1] {$b$}
node [right = 2pt of B_bot1, yshift = -1ex] {$y$}
;

\end{tikzpicture}
\end{center}

This wiring diagram shows that the two machines interact by using the output of each machine as the input to the other.  The composite machine has $S = 2$, $\Xin = \Xout = 0$ and an update function $(x,y) \mapsto (x + y, x)$ that takes no input. This composite will compute the Fibonacci sequence when initialized with the state $x = 1$, $y = 0$.
\end{example}

This example can be formalized using the operad algebra for composing multiple real-valued Moore machines defined in~\cite{DWD}. This operad algebra consists of:
\begin{itemize}
    \item An operad of directed wiring diagram, $\mathcal{O}(\DWD)$, which defines the interfaces for each system and the composition patterns between interfaces. 
    \item An algebra $\mathcal{O}(\DWD) \to \mathcal{O}(\Set)$ which defines the systems themselves and how they compose according to the patterns established by the operad.
\end{itemize}

In a Moore machine, the output is shielded from the input by the state. Not all models adhere to this constraint. Mealy machines, in particular, generalize this behavior by allowing the output to depend on \textit{both} the input and on the state. Consequently, a Mealy machine has a readout function with signature $\R^\Xin \times \R^S \to \R^\Xout$. In the special case where $S=0$, Mealy machines are simply functions $\R^\Xin \to \R^\Xout$.

\begin{example}\label{example:loop-example}

Consider the two Mealy machines both with $S = \Xin = \Xout = 1$ and with the following updates and readouts: 
\begin{align*}
    u_1(a,x) &= a + x & u_2(b,y) &= b\\
    r_1(a,x) &= a + x & r_2(b,y) &= b
\end{align*}
These Mealy machines are nearly identical to the Moore machines presented in Example~\ref{example:moore-composite} except the state is outputted instantaneously rather than at the next time step. Suppose that we attempt to compose these Mealy machines using the same wiring pattern as in Example~\ref{example:moore-composite}. This attempt is shown in the diagram on the left below. According to this wiring diagram, the output of each machine is the input to the other. Therefore, computing the output of the first machine requires the output of the second machine. But computing the output of the second machine requires the output of the first machine. And so forth, ad infinitum. The diagram on the right below contains the cycle between the wires in the wiring diagram (shown in purple) and the dependencies of output on input (shown in blue). This cycle implies  that the composite is not computable.

\begin{center}
    \begin{tikzpicture}[oriented WD, bbx = .5cm, bby =0.25cm, bb port length=4pt, bb port sep=.75, bb min width=1cm, bb min height=1cm]
    \node[bb={0}{0}{1}{1}] (A){$x = a + x$};
    \node[bb={0}{0}{1}{1}, right= of A] (B){$y = b$};
    \node[bb={0}{0}{0}{0}, fit={($(A.north west)+(0,1.25)$)($(B.south east)+(0, -3.5)$)}] (tot){};



    \draw[ar, color=violet] let \p1=(A.south west), \p2=(A.north west), \n1=\bbportlen, \n2=\bby in
     (B_bot1') to [out=-90, in=-90] (\x1 - \bby, \y1 - \n1) to [out=90, in=-90] (\x2 - \bby, \y2 +\n1) to [out=90, in=90] (A_top1);

    \draw[ar, color=violet] 
     let \p1=(A.south east), \p2=(B.north west), \n1=\bbportlen, \n2=\bby in
     (A_bot1') to [out=-90, in=-90] (\x1 + \bby, \y1 - \n1) to [out=90, in=-90] (\x2 - \bby, \y2 +\n1) to [out=90, in=90] (B_top1);

     \draw[label]
node [right = 2pt of A_top1] {$a$}
node [left = 2pt of A_bot1, yshift = -1ex] {$a + x$}
;

\draw[label]
node [right = 2pt of B_top1] {$b$}
node [right = 2pt of B_bot1, yshift = -1ex] {$b$}
;

\end{tikzpicture}
\hspace{1in} 
\begin{tikzpicture}[oriented WD, bbx = .5cm, bby =0.25cm, bb port length=4pt, bb port sep=.75, bb min width=1cm, bb min height=1cm]
    \node[bb={0}{0}{1}{1}] (A){};
    \node[bb={0}{0}{1}{1}, right= of A] (B){};
    \node[bb={0}{0}{0}{0}, fit={($(A.north west)+(0,1)$)($(B.south east)+(0, -3)$)}] (tot){};

        \draw[color=blue, dashed] (A_top1') to [out=-90, in=90] (A_bot1);

        \draw[color=blue, dashed] (B_top1') to [out=-90, in=90] (B_bot1);

    \draw[ar, color=violet] let \p1=(A.south west), \p2=(A.north west), \n1=\bbportlen, \n2=\bby in
     (B_bot1') to [out=-90, in=-90] (\x1 - \bby, \y1 - \n1) to [out=90, in=-90] (\x2 - \bby, \y2 +\n1) to [out=90, in=90] (A_top1);

    \draw[ar, color=violet] 
     let \p1=(A.south east), \p2=(B.north west), \n1=\bbportlen, \n2=\bby in
     (A_bot1') to [out=-90, in=-90] (\x1 + \bby, \y1 - \n1) to [out=90, in=-90] (\x2 - \bby, \y2 +\n1) to [out=90, in=90] (B_top1);

\end{tikzpicture}
\end{center}

\end{example}

This example shows that we can only compose Mealy machines according to wiring patterns that avoid cycles in the flow of information between component systems and from  input to output within each component system. This necessity for acyclicity is related to the acyclicity condition for causal models~\cite{pearl2000}. In Section~\ref{section:category} we formalize such wiring patterns by extending the operad of directed wiring diagrams introduced in~\cite{DWD} to an operad of \textit{dependent directed wiring diagrams}, in which interfaces track the dependency of output on input and the composition patterns require acyclicity between these dependencies and the wires between systems. Then in Section~\ref{section:dynamsection} we formalize the composition of Mealy machines according to these composition patterns as an algebra of this operad. The software package AlgebraicDynamics implements this composition of Mealy machines and can be found on GitHub at \url{https://algebraicjulia.github.io/AlgebraicDynamics.jl/}.

As in~\cite{DWD}, we will in fact construct a symmetric monoidal category of dependent directed wiring diagrams and a lax monoidal functor that defines the composition of Mealy machines. However, we often refer to their underlying operad and operad algebra to emphasize that this formalism can be used for composing multiple, independent systems. The operadic approach inherently accommodates $n$-ary operations, eliminating the need to break the operations down into binary operations.

In the tradition of systems dynamics, stock and flow diagrams (hereafter, called ``stock-flow diagrams'' for brevity) are used to graphically model certain classes of differential equations. In these models, stocks represent state variables and flows represent the transformation of one stock into another. Additionally, these models often contain (1) auxiliary variables, which are intermediate, derived quantities that are often of independent, domain importance and (2) links, which capture the instantaneous dependency of variables of stocks, on other auxiliary variables, and on exogenous variables. 

\cite{sfcomp} formalizes stock-flow diagrams and defines their composition by identifying the stocks of component systems, using the decorated cospan approach to composition~\cite{fong2015decorated}. In Section~\ref{section:SFsection}, we extend this formalization of stock-flow diagrams to include (1) links between auxiliary variables, (2) exogenous variables and links from exogenous variables to auxiliary variables, and (3) outputs. We also define alternative composition of stock-flow diagrams according to dependent directed wiring diagrams in which a stock-flow diagram's exogenous variables are set by the outputs of other diagrams. Finally, we show how to interpret a stock-flow diagram as a parameterized differential equation and prove that this interpretation is compatible with composition.

\paragraph{Acknowledgments} This material is based upon work supported by the Air Force Office of Scientific Research under award numbers FA9550-23-1-0376.

\section{Preliminaries}\label{section:prelims}

\subsection{Dependencies for functions of Euclidean spaces}\label{sec:functions}

The systems examined in this paper allow for the output to depend on both on the state and on the input into the system, creating dependencies between input and output ports. A natural framework for capturing these dependencies is by relations, represented as spans to indicate which outputs depend on which inputs. In this subsection, we provide background on dependencies for functions of Euclidean spaces.


Consider finite sets $A$ and $B$ and a map between them $f \colon A \to B$. We view an element of $\R^A$ to be a set function $x \colon A \to \R$. 
\begin{itemize}
    \item The pullback $f^* \colon  \R^B \to \R^A$ is defined by $f^*(x) \coloneqq x \circ f$.
    \item The pushforward $f_* \colon  \R^A\to\R^B$ is defined by  $f_*(x)(b) \coloneqq  \sum_{a\in f^{-1}(b)} x(a)$. 
\end{itemize}

\begin{definition}\label{definition:respectsrelation}

A function $\phi \colon  \R^A\to\R^B$ is said to respect a relation $ A \xleftarrow{s} R \xrightarrow{t} B$ if for all $b\in B$, the following diagram commutes
\[\begin{tikzcd}
	{\R^A} & {\R^B} \\
	{\R^{s( t^{-1}(b))}} & \R
	\arrow["\phi", from=1-1, to=1-2]
	\arrow["{i^*}"', from=1-1, to=2-1]
	\arrow["{\pi_b}", from=1-2, to=2-2]
	\arrow["{\phi_b}"', from=2-1, to=2-2]
\end{tikzcd}\]
where $\phi_b$ is the composite $\R^{s( t^{-1}(b))} \xrightarrow{i_*} \R^A \xrightarrow{\phi} \R^B \xrightarrow{\pi_b} \R$.
 
\end{definition}

Explicitly, $\phi$ respects the relation iff for all $x \in \R^A$ and $b \in B$, the value $\phi(x)(b) \in \R$ only depends on the components of $x$ that are related to $b$.

Note that if the relation $A \leftarrow R \rightarrow B$ is contained in the relation $A \leftarrow R' \rightarrow B$
then $\phi \colon \R^A \to \R^B$ respects $R$ implies that $\phi$ also respects $R'$.

\begin{restatable}[]{lemma}{compositeofrespectsrespects}\label{lem:composite-of-respects-respects}
If $f \colon \R^A \to \R^B$ respects $A \leftarrow R \to B$ and $g \colon \R^B \to \R^C$ respects $B \leftarrow S \to C$, then $ g\circ f \colon  \R^A \to \R^C$ respects the composite  span $A  \leftarrow R \times_B S \to C$.
\end{restatable}


    

\begin{notation}
For a span $ A \xleftarrow{s} R \xrightarrow{t} B$, we will write $R^* $ for the composite $t_* \circ s^* \colon \R^A \to \R^B$. For $x \in \R^A$ and $b \in B$, the value of $R^*(x)$ at $b$ is the sum of the values of $x$ at the the coordinates of $A$ to which $b$ is related.
\end{notation}

\begin{restatable}[]{lemma}{upperstarrespects}\label{lem:upper-star-respects}
    Given a relation $A \xleftarrow{s} R \xrightarrow{t} B$, the map $R^* \colon \R^A \to \R^B$ respects the relation.
\end{restatable}

\subsection{Dynamics on graphs}\label{sec:dynamics}

A graph $G = (V,E)$ consists of a set of vertices $V$, a set of edges $E$, and source and target maps $s, t \colon E \to V$. It will be useful to think of $G$ as the span $V \xleftarrow{s} E \xrightarrow{t} V$.  Given a graph $G$, the graph $\Path(G)$ has the same vertices as in $G$ and its edges are the paths of edges in $G$.

\begin{restatable}[]{lemma}{first}\label{lemma:first}

Let $G$ be a directed acyclic graph (DAG). Consider the endomorphism $\phi \colon \R^{G(V)}\to\R^{G(V)}$. If $\phi$ respects the relation $G(V) \xleftarrow{s} \Path(G)(E) \xrightarrow{t} G(V)$, then $\phi$  has a unique fixed point.

\end{restatable}







For a DAG $G = (V, E)$, suppose that we have a function $\phi \colon \R^V \to \R^V$ such that the value of $\phi$ at a vertex $v$ only depends on the values of the parents of $v$ (i.e., vertices with edges pointing directly into $v$). Then given a point $a \in \R^V$, we can push $a$ through $\phi$ inductively by $x_0 = a$ and $x_{i+1} = \phi(x_i) + a$. Then $x_{|V|}$ is in fact a fixed point of $\phi(-) + a$ and the value of $x_{|V|}$ at a vertex $v$ only depends on the value of $a$ at ancestors of $v$ (i.e., vertices from which there is a directed path to $v$). This observation is captured by the following lemma.

\begin{restatable}[]{lemma}{fpancestors}\label{lemma:fpancestors}

Let $G=(V,E)$ be an acyclic DAG, where $\phi \colon \R^V\to\R^V$ respects the relation $V \xleftarrow{s} E \xrightarrow{t} V$
In other words, for $v \in V$, $\pi_v \circ \phi \colon \R^V \to \R$ only depends on the parents of $v$.
Let $f \colon \R^V\to\R^V$ map $a \in \R^V$ to the unique fixed point of $\phi(-) + a \colon \R^V \to \R^V$. Then $f$ respects the relation $V \xleftarrow{s} \mathsf{Path}(G) \xrightarrow{t} V$.
\end{restatable}

\subsection{Directed wiring diagrams}\label{sec:dwd}
The symmetric monoidal category of directed wiring diagrams, denoted $\DWD$ consists of:

\begin{itemize}
\item Objects are pairs $\binomx$ where $\Xin$ is a finite set of input ports and $\Xout$ is a finite set of output ports.
\item  Morphisms $f \colon \binomx \to \binomy$ in $\DWD$ are diagrams of finite sets:

\begin{equation}\label{diagram:morphism}
\begin{tikzcd}[column sep=small, row sep = small]
	& {f(\Win)} && {f(\W)} && {f(\Wout)} \\
	\Yin && \Xin && \Xout && \Yout
	\arrow["s"', from=1-2, to=2-1]
	\arrow["t", from=1-2, to=2-3]
	\arrow["t"', from=1-4, to=2-3]
	\arrow["s", from=1-4, to=2-5]
	\arrow["s"', from=1-6, to=2-5]
	\arrow["t", from=1-6, to=2-7]
\end{tikzcd}
\end{equation}

The sets  $f(\Win)$, $f(\W)$, and $f(\Wout)$  represent input, trace, and output wires. The maps $s$ and $t$ denote the source and target maps.  We call such a morphism a \defined{directed wiring diagram}. 
\end{itemize}

Composition in $\DWD$ is defined as follows. For $f \colon \binomx \to \binomy$ and $g \colon \binomy \to \binomz$, their composite is
\begin{align*}
    (g\circ f)(\Win) &= g(\Win) \times_\Yin f(\Win)\\
    (g\circ f)(\W) &= f(\W) + f(\Wout) \times_\Yout g(\W) \times_\Yin f(\Win) \\
    (g \circ f)(\Wout) & = g(\Wout) \times_\Yout f(\Wout).
\end{align*}
The monoidal product in $\DWD$ is given by coproducts of finite sets, and in fact $\DWD$ is a cocartesian monoidal category.
\Cref{fig:exampledwd} shows an example of a directed wiring diagram.

\begin{figure}[htb]
    \centering
    \begin{tikzpicture}[oriented WD, bbx = .5cm, bby =0.25cm, bb port length=4pt, bb port sep=.75, bb min width=1cm, bb min height=1cm]
    \node[bb={0}{0}{2}{2}] (A){};
    \node[bb={0}{0}{1}{1}, right=of A] (B){};
    
    \node[bb={0}{0}{3}{3}, fit={($(B.north east)+(0,2)$)($(A.south west)+(0, -3.5)$)($(B.south west)+(0, 2)$)($(A.south east)+(0, -3.5)$)}] (tot){};
    
    \draw[ar, color=violet] (tot_top1') to [out=-90, in=90] (A_top1);

     \draw[ar, color=violet] (tot_top2') to [out=-90, in=90] (A_top2);

    \draw[ar, color=violet] (tot_top1') to [out=-90, in=90] (A_top1);

  \draw[ar, color=violet] (A_bot1) to [out=-90, in=90] (tot_bot1);

   \draw[ar, color=violet] (tot_top2) to [out=-90, in=90] (B_top1);

    \draw[ar, color=violet] (B_bot1) to [out=-90, in=90] (tot_bot3);
        
   \draw[ar, color=violet] (A_bot1) to [out=-90, in=90] (tot_bot3);

    \draw[ar, color=violet] let \p1=(A.south east), \p2=(A.north east), \n1=\bbportlen, \n2=\bby in (A_bot2') to [out=-90, in=-90] (\x1 + \bby, \y1 - \n1) to [out=90, in=-90] (\x2 + \bby, \y2 +\n1) to [out=90, in=90] (B_top1);

\end{tikzpicture}
    \caption{A directed wiring diagram $\binom{2}{2}  + \binom{1}{1} \to \binom{3}{3}$. There are three input wires, a single trace wire, and three output wires.}
    \label{fig:exampledwd}
\end{figure}
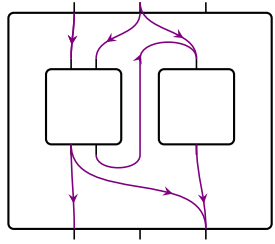

Libkind et al. \cite{DWD} defines a lax monoidal functor $F \colon \left(\DWD, + , \binom{0}{0}\right) \to (\Set, \times, 1)$ that defines the composition of Moore machines along directed wiring diagrams. On objects, $F$ maps an interface $\binomx$ to the set of Moore machines $(S \colon \Finset, r \colon \R^S \to \R^{\Xout} , u \colon \R^{\Xin} \times \R^S \to \R^S)$
where $S$ is the set of states, $r$ is the readout map, and $u$ is the update map. On morphisms, a directed wiring diagram $f \colon \binomx \to \binomy$ maps $(S, r, u)$ to the Moore machine with states $S$, readout map $f(\Wout)^* \circ r \colon \R^S \to \R^{\Yout}$
and update map taking $(a, x) \in \R^{\Yin} \times \R^S$ to 
\[
    u(f(\Win)^*(a) + f(W)^*(r(x)), x).
\]

Since $F$ is lax monoidal it has an underlying operad algebra. From the operadic perspective, given a directed wiring diagram 
\[
    f \colon \binom{X_{1, \text{in}}}{X_{1, \text{out}}} + \cdots  + \binom{X_{k, \text{in}}}{X_{k, \text{out}}} \to \binomy,
\] the set function
\[
    Ff \colon F\binom{X_{1, \text{in}}}{X_{1, \text{out}}} \times \cdots  \times F\binom{X_{k, \text{in}}}{X_{k, \text{out}}} \to F\binomy
\]
composes $k$ Moore machines with interfaces $\binom{X_{i, \text{in}}}{X_{i, \text{out}}}$ according to the directed wiring diagram $f$. The composite is a Moore machine with interface $\binomy$.

\section{Dependent directed wiring diagrams}\label{section:category}

In this section, we build on the category of directed wiring diagrams to define a category $\dDWD$ of dependent directed wiring diagrams. Just as the objects of $\DWD$ were interfaces for Moore machines, we will see in Section~\ref{section:dynamsection} and Section~\ref{section:SFsection} that the objects of $\dDWD$ will be interfaces of instantaneous systems including Mealy machines and stock-flow diagrams. These interfaces will track the dependencies of outputs on inputs and its morphisms will ensure that there are no cycles created by trace wires and these dependencies.

We  define the category $\dDWD$ in two steps. First, we define a lax monoidal functor $\dep \colon \DWD \to \Poset$ which maps an object $\binomx$ to the poset of relations between $\Xin$ and $\Xout$. If $\xin \in \Xin$ is related to $\xout \in \Xout$ then we say that the output $\xout$ depends on the input $\xin$. Second, we define $\dDWD$ to be the wide subcategory of the Grothendieck construction $\int \dep$ that only contains acyclic morphisms.


\begin{restatable}[]{proposition}{Grothendieckfunctor}
There is a symmetric lax monoidal functor $\dep \colon\left( \DWD, + , \binom{0}{0}\right) \to (\Poset, \times, 1)$ such that:

\begin{itemize}
    \item On objects, $\binomx$ maps to the poset of relations between $\Xin$ and $\Xout$. We often notate a relation by a span $\Xin \xfrom{s} d_X \xto{t} \Xout$ to emphasize that the relation represents a dependency of outputs on inputs. 

    \item On morphisms,  $f \colon \binomx \to \binomy$ in $\DWD$ maps to a map of posets $\dep(f) \colon  \dep\binomx\to \dep\binomy$ defined as follows. For each relation $d_X\in \dep(\binomx)$, the relation  $\dep(f)(d_X)$ is the relation induced by the span
\begin{equation}\label{eq:R(f)}
    \Yin \xfrom{s} f(\Win) \times_X \Path(f(W) +_X d_X) \times_X f(\Wout) \xto{t} \Yout
\end{equation}
where $X = \Xin + \Xout$ and $f(W) +_X d_X$ is the graph whose vertices are $X$ and whose edges are given by the trace wires $f(W)$ and the relation $d_X$. We name this graph $f(W) +_X d_X$ because it is the pushout in the category of graph of the graphs $f(W)$ and $d_X$ with apex the discrete graph on $X$. 
\end{itemize}

\end{restatable}

\begin{figure}[h]
\centering
\begin{subfigure}[b]{0.5\textwidth}
    \centering
     \begin{tikzpicture}[oriented WD, bbx = .5cm, bby =0.25cm, bb port length=4pt, bb port sep=.75, bb min width=1cm, bb min height=1cm]
    \node[bb={0}{0}{2}{2}] (A){};
    
    \node[bb={0}{0}{3}{3}, fit={($(A.north west)+(0,1.5)$)($(A.south east)+(0, -1.5)$)}] (tot){};
    
    \draw[ar, color=violet] (tot_top1') to [out=-90, in=90] (A_top1);

     \draw[ar, color=violet] (tot_top2') to [out=-90, in=90] (A_top2);

    \draw[ar, color=violet] (tot_top1') to [out=-90, in=90] (A_top1);  

    \draw[color=blue, dashed] (A_top1') to [out=-90, in=90] (A_bot1);

  \draw[color=blue, dashed] (A_top2') to [out=-90, in=90] (A_bot2);

  \draw[ar, color=violet] (A_bot1) to [out=-90, in=90] (tot_bot1);

    \draw[ar, color=violet] (A_bot2) to [out=-90, in=90] (tot_bot3);

         \draw[ar, color=violet] let \p1=(A.south east), \p2=(A.north east), \n1=\bbportlen, \n2=\bby in
     (A_bot1') to [out=-90, in=-90] (\x1 + \bby, \y1 - \n1) to [out=90, in=-90] (\x2 + \bby, \y2 +\n1) to [out=90, in=90] (A_top2);

\end{tikzpicture} 
    \vspace{-.3in}
    \subcaption{}
\end{subfigure}\hfill
\begin{subfigure}[b]{0.5\textwidth}
    \centering
     \begin{tikzpicture}[oriented WD, bbx = .5cm, bby =0.25cm, bb port length=4pt, bb port sep=.75, bb min width=1cm, bb min height=1cm]

    \node[bb={0}{0}{3}{3}, fit={($(A.north west)+(0,1.5)$)($(A.south east)+(0, -1.5)$)}] (tot){};

  \draw[color=blue, dashed] (tot_top1') to [out=-90, in=90] (tot_bot1);

   \draw[color=blue, dashed] (tot_top1') to [out=-90, in=90] (tot_bot3);

    \draw[color=blue, dashed] (tot_top2') to [out=-90, in=90] (tot_bot3);

\end{tikzpicture}
     \vspace{-.3in}
     \subcaption{}
     \end{subfigure}    
    \caption{(a) A directed wiring diagram $f \colon \binom{2}{2} \to \binom{3}{3}$ is depicted in purple. The blue dashed lines show a dependency $d_X \in \dep\binom{2}{2}$ where the first output port depends solely on the first input port and the second output port depends solely on the second input port. (b) In the resulting dependency $\dep(f)(d_X) \in \dep\binom{3}{3}$ the first output port depends on the first input port, the second output port has no dependencies, and the third output port depends on the first and second input ports.} 

\label{fig:fullfig}
\end{figure}
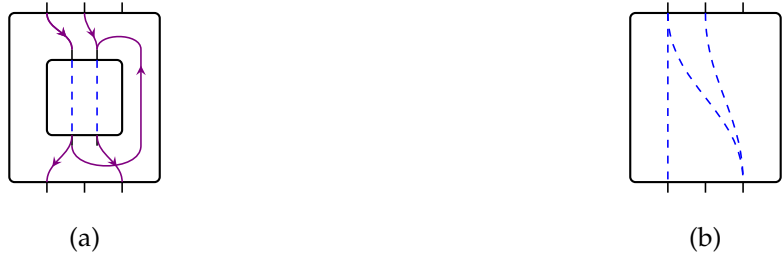

\begin{example}

\Cref{fig:fullfig} depicts the action of a directed wiring diagram $f \colon \binom{2}{2} \to \binom{3}{3}$ on a dependency  $d_X \in \dep\binom{2}{2}$.
Recall that the relation $\dep(f)(d_X)$ is induced by the span in Equation~\ref{eq:R(f)}. By inspection of the source and target maps, an element of this span is a tuple consisting of an input wire, a list alternating between dependencies and trace wires, and an output wire. Adjacent elements of this tuple must have matching sources and targets. 
The three dependencies depicted in \Cref{fig:fullfig}(b) from left to right consist of: 
\begin{itemize}
    \item The first input wire, the first dependency, and the first output wire.
    \item The first input wire, the first dependency, the single trace wire, the second dependency, and the second output wire.
    \item The second input wire, the second dependency, and the second output wire.
\end{itemize}
In this example, each element of the span defined a unique dependency. In general, this may not be the case.

\end{example}

Moreover, because pullbacks distribute over coproduct in the category of finite sets, we have the following lemma.

\begin{lemma}\label{lem:isospans}
    For morphisms $f \colon  \binomx \to \binomy$ and $f' \colon \binom{\Xin'}{\Xout'} \to \binom{\Yin'}{\Yout'}$ in $\DWD$ the spans
\[
    X + X' \leftarrow \Path((f+ f')(\W) +_{X + X'} (d_X + d_{X'}) )  \rightarrow X + X'
\]
and
\[
    X + X' \leftarrow \Path(f(\W) +_X d_X) + \Path(f'(\W) +_{X'} d_{X'}) \rightarrow X + X'
\] 
are isomorphic.
\end{lemma}



The Grothendieck construction of $R$ is the category $\int R$ whose:
\begin{itemize}
    \item Objects are pairs $\left( \binomx, d_X\right)$, where $\binomx \in \mathsf{Ob}(\DWD)$ and $d_X$ is a dependency on $\binomx$.
    \item Morphisms $f \colon \left(\binomx, d_X\right) \to \left(\binomy, d_Y\right)$ are directed wiring diagrams $f \colon \binomx \to \binomy$  such that  $\dep(f)(d_X) \leq d_Y$.
\end{itemize}

Since $\dep \colon \DWD \to \Poset$ is symmetric lax monoidal, the category $\int \dep$ inherits a symmetric monoidal structure~\cite{grothendieck}. Furthermore, since $\DWD$ is cocartesian, this monoidal structure is unambiguous~\cite[Theorem 4.2]{grothendieck}. It is straightforward to check that this monoidal structure is in fact the coproduct.

As we saw in Example~\ref{example:loop-example}, cycles between the trace wires of $f$ and the dependencies in $\binomx$ are not compatible with composing the dynamics of systems where the output may instantaneously depend on the input.

\begin{definition}\label{definition:cycle}
Let  $f \colon \left(\binomx, d_X\right) \to \left(\binomy, d_Y\right)$ be a morphism in $\int \dep$. As before let $X = \Xin + \Xout$ and $f(W) +_X d_X$ be the graph whose vertices are $X$ and whose edges are $f(W) + d_X $. Then we say $f$ is \defined{(a)cyclic} if the graph $f(W) +_X d_X\rightrightarrows X$  is (a)cyclic. 
\end{definition}

\begin{restatable}[]{lemma}{compositionofacyclic}\label{lem:composition-of-acyclic}

The composition of acyclic morphisms in $\int \dep$ is acyclic.

\end{restatable}

By Lemma~\ref{lem:composition-of-acyclic} and the fact that identity morphisms are acyclic, the acyclic morphisms of $\int \dep$ form a wide subcategory, which we call the category of dependent directed wiring diagrams, denoted $\dDWD$.


By Lemma~\ref{lem:isospans} and the fact that the coproduct of two acyclic graphs is again acyclic, the monoidal product of two acyclic morphisms is again acyclic. Furthermore, the associator, unitors, and braiding are acyclic as well. Therefore, $\dDWD$ inherits a symmetric monoidal structure from $\int \dep$.

\section{Composing Mealy machines}\label{section:dynamsection}

Here, we  define an algebra $\Dynam \colon \dDWD\to \Set$ for composing real-valued Mealy machines. 

\begin{definition}
    Given an object $\left(\binomx, d_X\right)$ in $\dDWD$, define $\Dynam\left(\binomx, d_X\right)$  to be the set of triples 
    \[
        (S \colon \Finset, u \colon \R^{\Xin} \times \R^S \to \R^S, r \colon \R^{\Xin} \times \R^S \to \R^{\Xout}),
    \] where $S$ is a finite set of states, $u$ is an update map, and $r$ is a readout map that respects the relation $\Xin \xfrom{s} d_X \xto{t} \Xout$. One such triple is called a \defined{Mealy machine}.
\end{definition}



Let $f \colon \left(\binomx, d_X\right)\to \left(\binomy, d_Y\right)$ be a morphism in $\dDWD$ and let $r$ be the readout of a Mealy machine in $\Dynam\left(\binomx, d_X \right)$. Given a state $s \in \R^S$ and $a \in \R^\Xin$, we want to determine values for both inputs and outputs that remain consistent when signal flows --- via the readout and trace wires ---  through the system. We do this by computing the fixed point of an endomorphism that captures these signal flows.

Define an endomorphism $\phi^{f,r}(a, s) \colon \R^\Xin \times \R^\Xout \to \R^\Xin \times \R^\Xout$  by 
\begin{equation}\label{eq:xin-xout-endomorphism}
    \phi^{f,r}(a,s)(\xin, \xout) \coloneqq (f(\W)^*(\xout) + a, r(\xin, s)).
\end{equation} 

Because $f$ is acyclic, $\phi^{f,r}(a,s)$ satisfies the hypotheses of Lemma~\ref{lemma:first} and hence has a unique fixed point $(\oxin, \oxout)$.

\begin{definition}\label{def:oxout-f-r}
    For $f \colon \left(\binomx, d_X\right)\to \left(\binomy, d_Y\right)$ in $\dDWD$ and $(S, u, r) \in \Dynam\left(\binomx, d_X\right)$, define 
    \[
        \ophi^{f,r} = (\ophiin^{f,r}, \ophiout^{f,r}) \colon \R^{\Xin} \times \R^S \to \R^\Xin \times \R^\Xout
    \]
    such that $\ophi^{f,r}(a,s)$ is the unique fixed point of $\phi^{f,r}(a,s)$.
\end{definition}

\begin{example}
    
Consider the morphism $f\colon \left(\binom{2}{2}, d_X\right) \to \left(\binom{2}{1}, d_Y\right)$ pictured below where $d_X$ is the identity span $2 \leftarrow 2 \to 2$ and $d_Y$ is the span $2 \leftarrow 2 \to 1$ where the left leg is the identity.

Consider the  triple \(
    (0, ! \colon \R^2 \times \R^0 \to \R^0, r \colon \R^2 \times \R^0 \to \R^2).
\)
The readout map $r$ respects the relation $d_X$ if and only if there exists maps $r_1, r_2 \colon \R \to \R$ such that 
\(
    r((a_1, a_2), *) = (r_1(a_1), r_2(a_2))   
\).
In such cases, $(0, !, r)$ is a Mealy machine in $\Dynam\left(\binom{2}{2}, d_X\right)$.

Given an input $a = (a_1, a_2) \in \R^2$, the fixed point of $\phi^{f,r}(a,*)$ consists of an input
\(
    \ophiin^{f,r}(a,*) = (a_1, r_1(a_1) + a_2) \in \R^2
\) and an output
\(
    \ophiout^{f,r}(a, *) = (r_1(a_1), r_2(r_1(a_1) + a_2)) \in \R^2
\).
The diagram below show how this fixed point propagates the input through readout and feedback wires. The purple arrows indicate the flow of values, with the output $r_1(a_1)$ feeding back into the second input as part of $r_1(a_1) + a_2$.
\begin{center}
     \begin{tikzpicture}[oriented WD, bbx = .5cm, bby =0.25cm, bb port length=4pt, bb port sep=.75, bb min width=1cm, bb min height=1cm]
    \node[bb={0}{0}{2}{2}] (A){};
    
    \node[bb={0}{0}{2}{1}, fit={($(A.north west)+(0,2)$)($(A.south east)+(0, -2)$)}] (tot){};
    
    \draw[ar, color=violet] (tot_top1') to [out=-90, in=90] (A_top1);
    \draw[ar, color=violet] (tot_top2') to [out=-90, in=90] (A_top2);
    \draw[ar, color=violet] (tot_top1') to [out=-90, in=90] (A_top1);  
    \draw[color=blue, dashed] (A_top1') to [out=-90, in=90] (A_bot1);
    \draw[color=blue, dashed] (A_top2') to [out=-90, in=90] (A_bot2);
    \draw[ar, color=violet] (A_bot2) to [out=-90, in=90] (tot_bot1);
    \draw[ar, color=violet] let \p1=(A.south east), \p2=(A.north east), \n1=\bbportlen, \n2=\bby in
     (A_bot1') to [out=-90, in=-90] (\x1 + \bby, \y1 - \n1) to [out=90, in=-90] (\x2 + \bby, \y2 +\n1) to [out=90, in=90] (A_top2);
  
    \draw[label]
      node [left = 2pt of A_top1] {$a_1$}
      node [left = 2pt of A_bot1, yshift = -1ex] {$r_1(a_1)$}
    ;
    \draw[label]
      node [right = 2pt of A_top2] {$r_1(a_1) + a_2$}
      node [right = 2pt of tot_bot1, yshift = -1ex] {$r_2(r_1(a_1)+a_2)$}
    ;
    
    \draw[label]
      node [above = 2pt of tot_top1] {$a_1$}
      node [above = 2pt of tot_top2] {$a_2$}
    ;
\end{tikzpicture}

\end{center}

\end{example}



Let  $f \colon \left(\binomx, d_X\right) \to \left(\binomy, d_Y\right)$ be a morphism in $\dDWD$ and let $(S, u, r) \in \Mealy\left(\binomx, d_X\right)$. Define $\Mealy(f)(u)$ to be the composite
\[
    \Mealy(f)(u) \colon \R^{\Yin} \times \R^S \xrightarrow{f(\Win)^* \times \triangle} \R^{\Xin} \times \R^S \times \R^S \xrightarrow{\ophiin^{f,r} \times \R^S} \R^{\Xin} \times \R^S \xrightarrow{u} \R^S
\]
and define $\Mealy(f)(r)$ to be  the composite
\[
    \Mealy(f)(r) \colon \R^{\Yin} \times \R^S \xrightarrow{f(\Win)^* \times \R^S} \R^{\Xin} \times \R^S \xrightarrow{\ophiout^{f,r}} \R^{\Xout} \xrightarrow{f(\Wout)^*} \R^{\Yout}.
\]
Then define 

$$\Mealy(f) \colon \Mealy\left(\binomx, d_X\right) \to \Mealy\left(\binomy, d_Y\right)$$   

$$(S, u, r) \mapsto (S, \Mealy(f)(u), \Mealy(f)(r)).$$





\begin{restatable}[]{theorem}{mealylaxmonoidal}

$\Dynam \colon \left(\dDWD, + , \left(\binom{0}{0}, ! \right) \right) \to (\Set, \times, 1)$ is a lax monoidal functor.
\end{restatable}

\section{Stock-flow diagrams}\label{section:SFsection}

A stock-flow diagram graphically models a system in which flows transform one stock into another. For example, in the simple stock-flow diagram below there are three stocks representing a susceptible, infective, and recovered population. Three flows transform susceptible people into infective, infective into recovered, and recovered back into susceptible. Two one-sided flows represents the birth of susceptible individuals and the death of recovered individuals.

\begin{figure}[thb]
    \centering
        \centering
        \includegraphics[width=0.65\linewidth,scale=1.5]{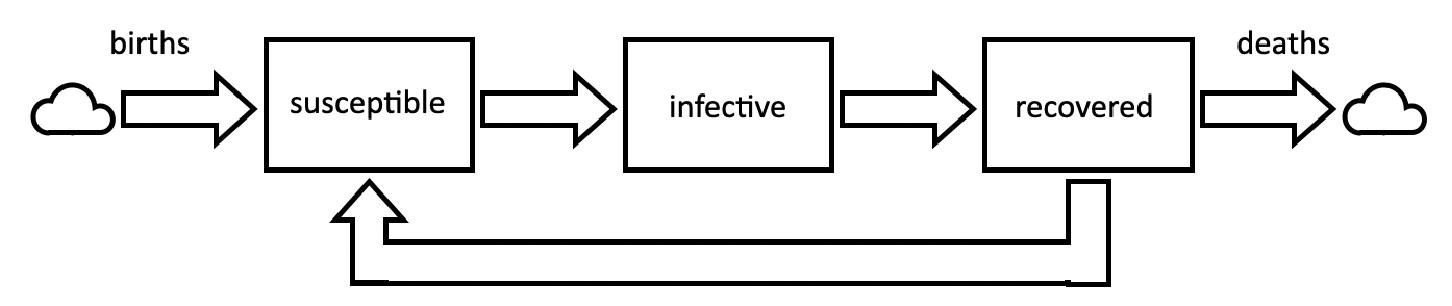}
\end{figure}


\subsection{Composing stock-flow diagrams}
In this subsection, we construct an operad algebra $\SF \colon \mathcal{O}(\dDWD) \to \mathcal{O}(\Set)$ for composing  stock-flow diagrams. We focus on full-fledged, open stock-flow diagrams, which consist of stocks, flows, sum variables, (auxiliary) variables,  in ports, and out ports. 
Each flow is assigned a variable that defines its rate. Each variable is linked to stocks, sum variables, variables, and in ports. The value of each variable is given by the auxiliary function and may only depend on the components to which it is linked. Finally, each out port outputs the sums of the variables to which it is linked. This setup is captured by the following definition.

\begin{definition}\label{definition:opensf}

An open stock-flow diagram with interface $\left(\binomx, d_X \right) \in \ob(\dDWD)$ consists of:

\begin{itemize}
    \item  A functor $F \colon \Hf \to \Set$ where $\Hf$ is the free category of the diagram 
\[\begin{tikzcd}[row sep = small]
	&& \stocksumlink & \sumvar \\
	& \stock & \stocklink & \sumlink \\
	\inflow && \outflow & \aux & \inlink & \inport \\
	& \flow && \outlink & \varlink \\
	&&& \outport
	\arrow["t", from=1-3, to=1-4]
	\arrow["s"', from=1-3, to=2-2]
	\arrow["s"', from=2-3, to=2-2]
	\arrow["t"', from=2-3, to=3-4]
	\arrow["s"', from=2-4, to=1-4]
	\arrow["t"', from=2-4, to=3-4]
	\arrow["{\text{is}}", from=3-1, to=2-2]
	\arrow["{\text{if}}"', from=3-1, to=4-2]
	\arrow["{\text{os}}"', from=3-3, to=2-2]
	\arrow["{\text{of}}", from=3-3, to=4-2]
	\arrow["t"', from=3-5, to=3-4]
	\arrow["s", from=3-5, to=3-6]
	\arrow["{\text{fv}}"', curve={height=18pt}, from=4-2, to=3-4]
	\arrow["s", from=4-4, to=3-4]
	\arrow["t"', from=4-4, to=5-4]
	\arrow["t", shift left, from=4-5, to=3-4]
	\arrow["s"', shift right, from=4-5, to=3-4]
\end{tikzcd}\]

such that (1) $F(\text{if})$ and $F(\text{of})$ are injective, (2) the graph whose edges are $F(\varlink)$ and whose vertices are $F(\aux)$ is acyclic,  (3) $F(\inport) = \Xin$ and $F(\outport) = \Xout$, and (4) the relation $d_X$ contains the relation induced by the span
   \[
        \Xin = F(\inport) \leftarrow F(\inlink) \times_{F(\aux)} \Path(F(\varlink)) \times_{F(\aux)} F(\outlink) \to F(\outport) = \Xout
   \] where $\Path(F(\varlink))$ is the path graph of the graph $F(\varlink) \rightrightarrows F(\aux)$.

    \item An auxiliary function $\varphi \colon  \Rst\times\Rsv\times\Rv\times\R^{F(\inport)}\to\Rv$ such that $\varphi$ respects the span
    
        \begin{equation}\label{eq:sf-respect-span}
    \begin{tikzcd}
	& {} \\
	{F(\stock) + F(\sumvar) + F(\aux) + F(\inport)} \\
	{F(\stocklink) + F(\sumlink) + F(\varlink) + F(\inlink)} \\
	{F(\aux)}
	\arrow["s", from=3-1, to=2-1]
	\arrow["t"', from=3-1, to=4-1]
\end{tikzcd}
\end{equation}
    
\end{itemize}
    
\end{definition}


\begin{example}\label{ex:initial-sir-example}

Below is an example of an open stock-flow diagram representing an SIR model, inspired by Figure 1 of~\cite{sfcomp}.
\begin{center}
    \includegraphics[width=0.65\linewidth]{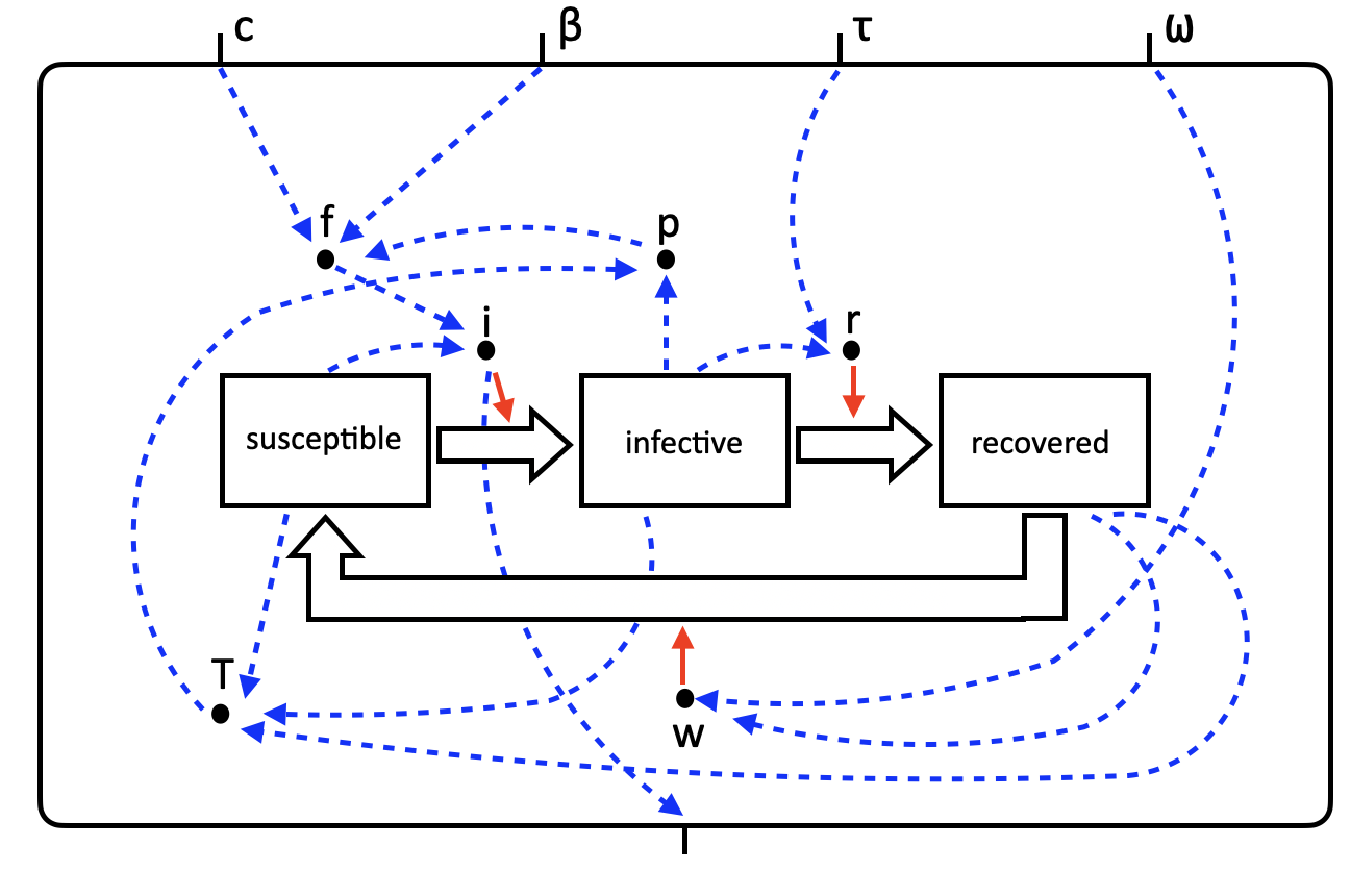}
\end{center}

This open stock-flow diagram has four in ports corresponding to contacts per unit time ($c$), per contact likelihood of infection ($\beta$), mean duration of infectiousness ($\tau$), and rate of waning of immunity ($\omega$).
It has five variables corresponding to force of infection ($f$), the fractional prevalence of infection ($p$),  infection ($i$),  recovery ($r$), and waning immunity ($w$).
It has one sum variable corresponding to the total population ($T$). 
And it has one out port which outputs the infection.

Links (shown in blue) depict the dependencies between the variables and other components. For example, referring back to \Cref{definition:opensf}, one can observe a variable link from $f$ to $i$, an input link from $\tau$ to $r$, and a stock link from infective to $p$. The assignment of each flow to the variable that defines its rate is shown in red.
To define the auxiliary function it suffices to give for each variable, a function of the components to which it is linked, which we do as follows: 
(1) The force of infection is the product  of the contact per unit time, per contact likelihood, and fractional prevalence. 
(2) The fractional prevalence is the infective population divided by the total population. 
(3) The infection is the product of the susceptible population with the force of infection. 
(4) The recovery is the product of the infective population with the mean duration of infectiousness. 
(5) The  waning of immunity is the product of the recovered population with the  waning immunity. In summary,
\[
f = c \beta p, \quad p = \frac{I}{T}, \quad i = Sf, \quad r = I\tau, \quad w = R\omega.
\]
Note that we do not give a function for the sum variable representing the total population, because it will  be computed  as the sum of the stocks to which it is linked.

The interface for this stock-flow diagram tracks that the single output depends on the first two inputs. 

\end{example}





Let $f \colon \left(\binomx, d_X \right) \to \left(\binomy, d_Y \right)$ in $\dDWD$, and let $(F, \varphi)$ be an open stock-flow diagram with  interface $\left(\binomx, d_X \right)$. Then define $\SF(f)(F) \colon \Hf \to \Set$ to be the diagram which is identical to $F$ on all objects and morphisms except for the following links and their source and target maps.
\begin{align*}
    \SF(f)(F)(\inlink) & = f(\Win) \times_\Xin F(\inlink)\\
    \SF(f)(F)(\outlink)  & = F(\outlink) \times_\Xout f(\Wout)\\
    \SF(f)(F)(\varlink) & = F(\varlink) + \left(F(\outlink) \times_\Xout f(W) \times_{\Xin}  F(\inlink)\right).
\end{align*}
We also define $\SF(f)(\varphi)  \colon \Rst\times\Rsv\times\Rv\times\R^{\Yin}\to\Rv$ to be the function that takes  $(x_\stock, x_\sumvar, x_\aux, \yin) \in \Rst\times\Rsv\times\Rv\times\R^{\Yin}$ to 
\[
      \varphi\left(x_\stock, x_\sumvar, x_\aux, f(\Win)^*(\yin) + (f(W)^* \circ  F(\outlink)^*)(x_\aux)\right).
\]

\begin{restatable}[]{proposition}{SFlaxmonoidal}
    There is a lax monoidal functor  $\SF \colon \left(\dDWD, + \left(\binom{0}{0}, !\right) \right) \to (\Set, \times, 1)$ which on objects takes $\left(\binomx, d_X \right)$ to the set of all open stock-flow diagrams with this interface. On morphisms $f$, it takes the open stock-flow diagram $(F, \varphi)$ to the open stock-flow diagram $(\SF(f)(F), \SF(f)(\varphi))$.
\end{restatable}

\begin{example}

    The stock-flow diagram in \Cref{fig:water-pollutant-sf}(a) has a single stock of water that grows at a constant rate and decreases at a rate proportional to the size of the water supply. This proportion is given by the single parameter to the system. These inflows and outflows are outputted by the system along with the water supply itself.
    The stock-flow diagram in \Cref{fig:water-pollutant-sf}(b) has a single stock of pollutants that grows in proportion to the inflow of its substrate. This proportion is given as a parameter to the system. The total pollutants decreases proportional to the outflow of its substrate. This proportion is given by the average pollutant per unit substrate, and this average is calculated by dividing the total pollutant by the size of the stock of substrate. 

    When we compose these stock-flow diagrams according to the dependent directed wiring diagram pictured in ~\Cref{fig:water-pollutant-ddwd}(a), we get a stock-flow diagram with a stock for the water supply and a stock for the pollutants in which the flow of the pollutants depends on the flow of the water. This composite is an example of the generic co-flow structure defined in Figure 12-17 of \cite{sterman2009businessdynamics}.
    
\begin{figure}[thb]
    \centering
    \begin{subfigure}[b]{0.45\textwidth}
        \centering
        \includegraphics[width=\textwidth]{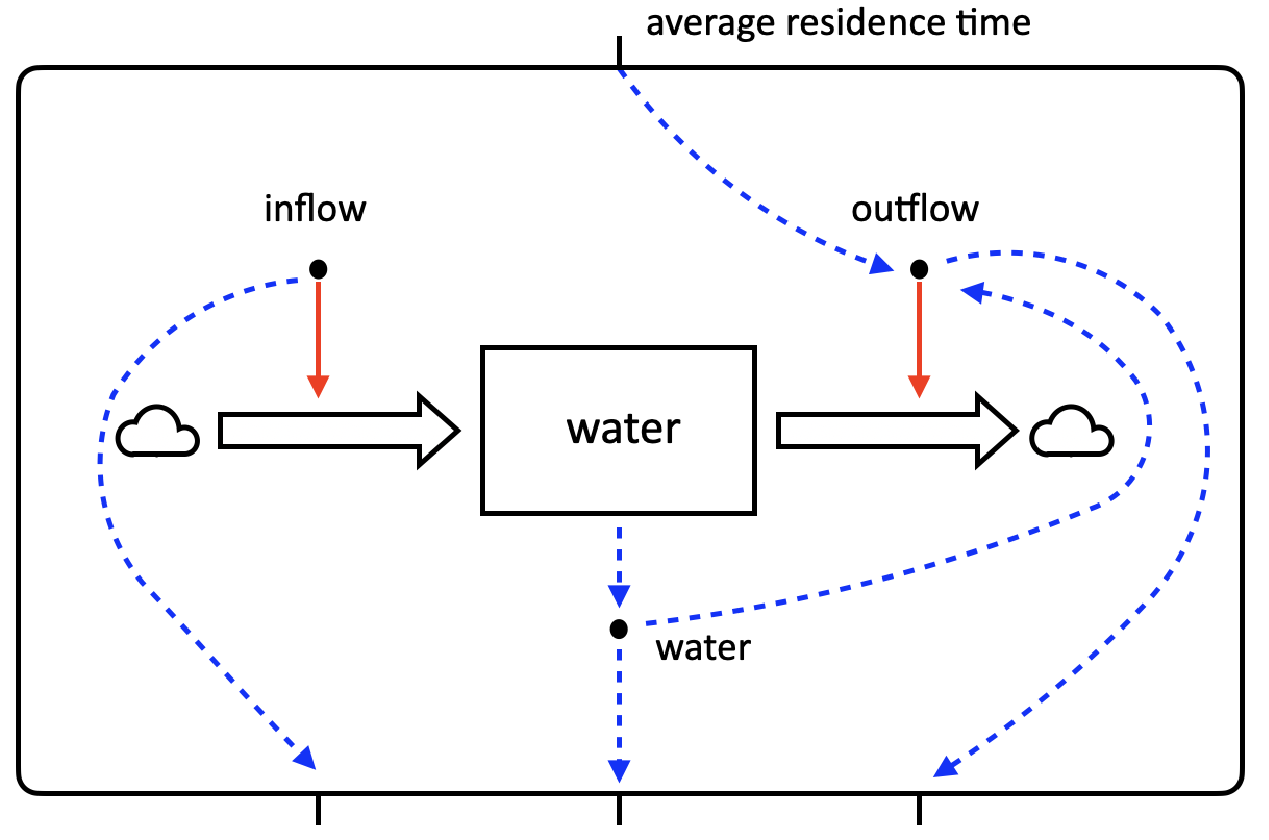}
        \caption{}
    \end{subfigure}
    \hfill
    \begin{subfigure}[b]{0.45\textwidth}
        \centering
        \includegraphics[width=\textwidth]{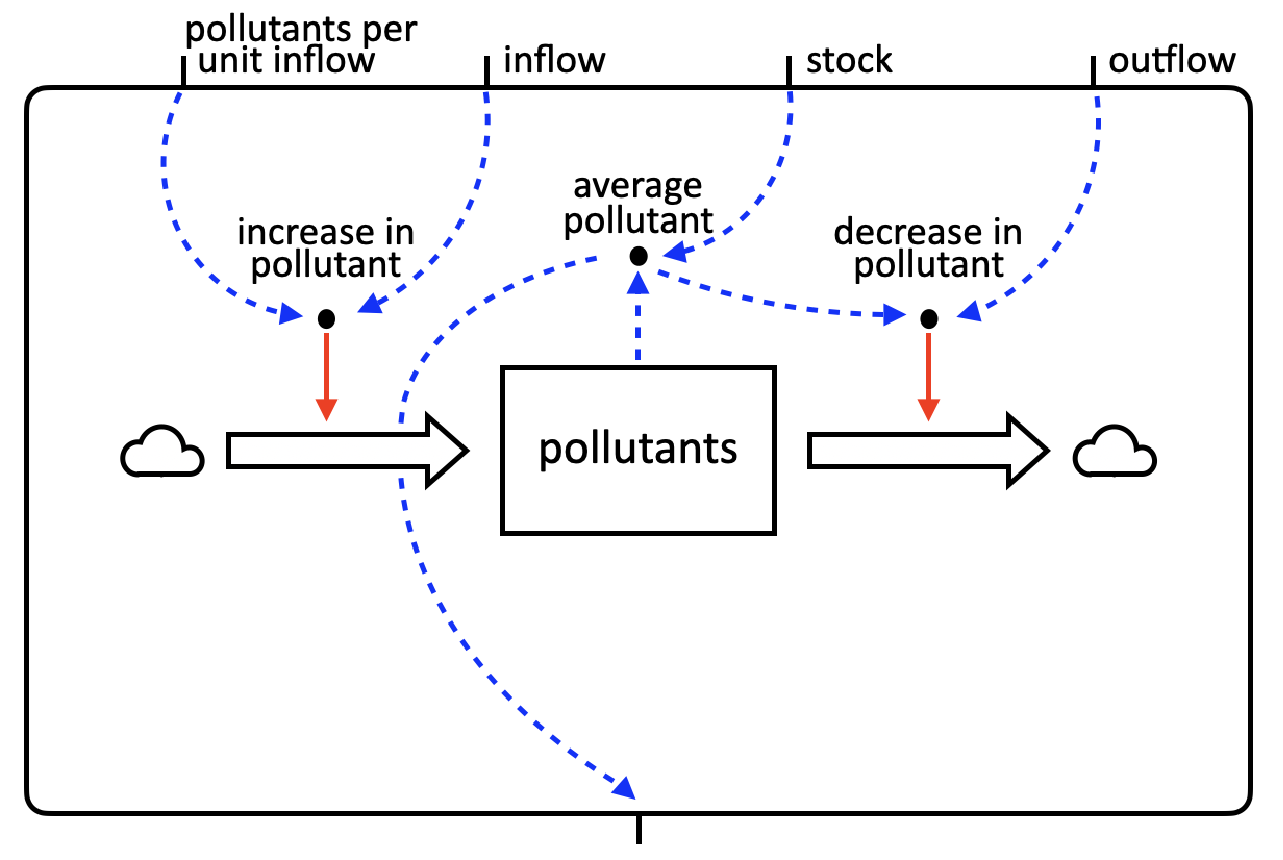}    \end{subfigure}
    \caption{Stock-flow diagrams representing  a changing (a) water supply and (b) concentration of pollutants.}
    \label{fig:water-pollutant-sf}
\end{figure}

    In the compositional approach to modeling,  the syntax for composition (the dependent directed wiring diagram) is independent from the models being composed. For example, we can compose the stock-flow diagrams in \Cref{fig:water-pollutant-sf} according to other composition patterns.  \Cref{fig:water-pollutant-ddwd}(b) defines a composition pattern for two pollutants co-flowing with a single water supply. Conversely, \Cref{fig:water-pollutant-ddwd}(c) defines a composition pattern for a pollutant co-flowing with two independent water supplies.

\begin{figure}[hbt]
   \centering
   \begin{subfigure}[b]{0.3\textwidth}
       \centering
       \begin{tikzpicture}[oriented WD, bbx = .5cm, bby =0.25cm, bb port length=4pt, bb port sep=.75, bb min width=1cm, bb min height=1cm]
    \node[bb={0}{0}{1}{3}] (A){};
    \node[bb={0}{0}{4}{1}, below=of A, yshift=-0.5cm] (B){};
    
    \node[bb={0}{0}{2}{1}, fit={($(A.north east)+(0,2)$)($(A.north west)+(0, 1)$)($(B.south west)+(0, 2)$)($(B.south east)+(0, -1)$)}] (tot){};

 \draw[ar, color=violet] (tot_top2) to [out=-90, in=90] (A_top1);
 
 \draw[ar, color=violet] (A_bot1') to [out=-90, in=90] (B_top2); 

  \draw[ar, color=violet] (A_bot2') to [out=-90, in=90] (B_top3);
  
  \draw[ar, color=violet] (A_bot3') to [out=-90, in=90] (B_top4);

 \draw[ar, color=violet] (B_bot1') to [out=-90, in=90] (tot_bot1);

  \draw[color=blue, dashed] (A_top1') to [out=-90, in=90] (A_bot3);

  \draw[color=blue, dashed] (B_top3') to [out=-90, in=90] (B_bot1);

\draw[ar, color=violet] let \p1=(A.north west), \p2=(A.south west), \n1=\bbportlen in
    (tot_top1') to[out=-90, in=90] (\x1 - \n1, \y1)  to[out=-90, in=90] (\x2 - \n1, \y2) 
    to[out=-90, in=90] (B_top1);

\end{tikzpicture}
       \vspace{-0.4in}
       \caption{}
    \end{subfigure}
   \hfill
    \begin{subfigure}[b]{0.3\textwidth}
       \centering
        \begin{tikzpicture}[oriented WD, bbx = .5cm, bby =0.25cm, bb port length=4pt, bb port sep=.75, bb min width=1cm, bb min height=1cm]
    \node[bb={0}{0}{1}{3}] (A){};
    \node[bb={0}{0}{4}{1}, below =of A, yshift=-0.5cm] (D){};
    \node[bb={0}{0}{4}{1}, left=0.25cm of D] (B){};
    
    \node[bb={0}{0}{3}{2}, fit={($(A.north east)+(0,2)$)($(A.north west)+(0, 1)$)($(B.south west)+(0, 2)$)($(D.south east)+(0, -1)$)}] (tot){};
    
    \draw[ar, color=violet] (tot_top3) to [out=-90, in=90] (A_top1);
    
    \draw[ar, color=violet] (A_bot1') to [out=-90, in=90] (B_top2); 
    \draw[ar, color=violet] (A_bot2') to [out=-90, in=90] (B_top3);
    \draw[ar, color=violet] (A_bot3') to [out=-90, in=90] (B_top4);
    
    \draw[ar, color=violet] (A_bot1') to [out=-90, in=90, looseness=0.7] (D_top2); 
    \draw[ar, color=violet] (A_bot2') to [out=-90, in=90, looseness=0.7] (D_top3);
    \draw[ar, color=violet] (A_bot3') to [out=-90, in=90, looseness=0.7] (D_top4);
    
    \draw[ar, color=violet] (B_bot1') to [out=-90, in=90] (tot_bot1);
    \draw[ar, color=violet] (D_bot1') to [out=-90, in=90] (tot_bot2);
    
    \draw[color=blue, dashed] (A_top1') to [out=-90, in=90] (A_bot3);
    \draw[color=blue, dashed] (B_top3') to [out=-90, in=90] (B_bot1);
    \draw[color=blue, dashed] (D_top3') to [out=-90, in=90] (D_bot1);
    
    \draw[ar, color=violet] let \p1=(A.north west), \p2=(A.south west), \n1=\bbportlen in
        (tot_top2') to[out=-90, in=90] (\x1 - \n1, \y1)  to[out=-90, in=90] (\x2 - \n1, \y2) 
    to[out=-90, in=90](D_top1);
        
    \draw[ar, color=violet] 
        (tot_top1') to[out=-90, in=90]   (B_top1);
\end{tikzpicture}
        \vspace{-0.4in}
        \caption{}
    \end{subfigure}
   \hfill
    \begin{subfigure}[b]{0.3\textwidth}
       \centering
        \begin{tikzpicture}[oriented WD, bbx = .5cm, bby =0.25cm, bb port length=4pt, bb port sep=.75, bb min width=1cm, bb min height=1cm]
    \node[bb={0}{0}{1}{3}] (A){};
    \node[bb={0}{0}{1}{3}, right=0.25cm of A] (C){};
    \node[bb={0}{0}{4}{1}, below=of A,  yshift=-0.5cm] (B){};
    
    \node[bb={0}{0}{3}{1}, fit={($(A.north west)+(0,2)$)($(C.north east)+(0, 1)$)($(B.south west)+(0, 2)$)($(B.south east)+(0, -1)$)}] (tot){};
    
    \draw[ar, color=violet] (tot_top2) to [out=-90, in=90] (A_top1);
    \draw[ar, color=violet] (tot_top3) to [out=-90, in=90] (C_top1);
    
    \draw[ar, color=violet] (A_bot1') to [out=-90, in=90] (B_top2); 
    \draw[ar, color=violet] (A_bot2') to [out=-90, in=90] (B_top3);
    \draw[ar, color=violet] (A_bot3') to [out=-90, in=90] (B_top4);
    
    \draw[ar, color=violet] (C_bot1') to [out=-90, in=90] (B_top2); 
    \draw[ar, color=violet] (C_bot2') to [out=-90, in=90] (B_top3);
    \draw[ar, color=violet] (C_bot3') to [out=-90, in=90] (B_top4);
    
    \draw[ar, color=violet] (B_bot1') to [out=-90, in=90] (tot_bot1);
    
    \draw[color=blue, dashed] (A_top1') to [out=-90, in=90] (A_bot3);
    \draw[color=blue, dashed] (C_top1') to [out=-90, in=90] (C_bot3);
    \draw[color=blue, dashed] (B_top3') to [out=-90, in=90] (B_bot1);
    
    \draw[ar, color=violet] let \p1=(A.north west), \p2=(A.south west), \n1=\bbportlen in
    (tot_top1') to[out=-90, in=90] (\x1 - \n1, \y1)  to[out=-90, in=90] (\x2 - \n1, \y2) 
    to[out=-90, in=90] (B_top1);

\end{tikzpicture}
        \vspace{-0.4in}
       \caption{}
    \end{subfigure}
    \caption{Dependent directed wiring diagrams for composing the stock-flow diagrams in \Cref{fig:water-pollutant-sf}.}
    \label{fig:water-pollutant-ddwd}
\end{figure}
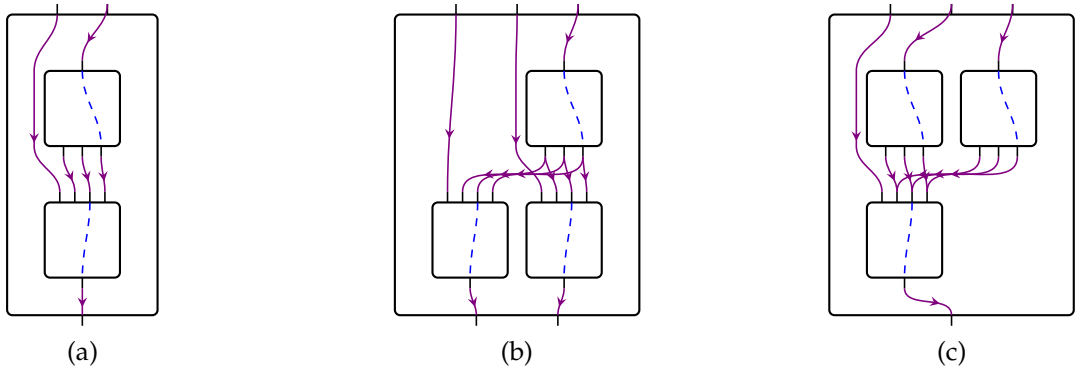

\end{example}



\subsection{Stock-flow diagrams to Mealy machines}\label{section:SFnattrans}

In this subsection, we interpret a diagrammatic stock-flow model  as a differential equation by defining a natural transformation $\alpha \colon  \SF\ \Rightarrow \Mealy$ and interpreting a Mealy machine's update $u \colon \R^{\Xin} \times \R^S \to \R^S$ as giving a tangent vector in $T \R^S$ rather than as an updated state.

\begin{definition}
Consider an open  stock-flow diagram  $(F, \varphi) \in \SF\left(\binomx, d_X \right)$.  By definition $\varphi \colon \Rst \times \Rsv \times \Rv \times \R^{\Xin} \to \Rv$ respects the span in \Cref{eq:sf-respect-span} and the graph with edges $F(\varlink)$ and vertices $F(\variable)$ is acyclic. Thus, for $s \in \R^{F(\stock)}$ and $a \in \R^{\Xin}$, the endomorphism
\[
    \varphi(s, F(\sumlink)^*(s), -, a) \colon \Rv \to \Rv
\] has a unique fixed point in $\Rv$ which we denote $\ovarphi(a,s)$. This defines a function $\ovarphi \colon \R^{\Xin} \times \R^{F(\stock)} \to \Rv$. 
\end{definition}

For $a \in \R^\Xin$ and $s \in \R^{F(\stock)}$, the fixed point $\ovarphi(a,s)$ assigns a value to each variable according to the auxiliary function given the input $a$ and the supply of each stock $s$.

\begin{definition}\label{def:alpha}

Let $\alpha \colon \SF\Rightarrow\Dynam$ be defined on components as follows.

For $\left(\binomx, d_X \right)$, define the set function $\alpha_X \colon \SF\left(\binomx, d_X \right) \to  \Dynam\left(\binomx, d_X \right)$ to take a stock-flow diagram $(F, \varphi) \in \SF\left(\binomx, d_X \right)$ to the Mealy machine whose:
\begin{itemize}
    \item States are the stocks, $F(\stock)$.
    \item Read out function $r(F, \varphi) \colon \R^{\Xin} \times \R^{F(\stock)} \to \R^\Xout$ is the composite $F(\outlink)^* \circ \ovarphi$.
    \item Update function $u(F, \varphi) \colon \R^{\Xin} \times \R^{F(\stock)} \to \R^{F(\stock)}$ is the defined by
    \[
        u(a,s) = (F(\text{is})_*\circ F(\text{if})^* \circ F(\text{fv})^*) ( \ovarphi(a,s)) - (F(\text{os})_*\circ F(\text{of})^* \circ F(\text{fv})^*) ( \ovarphi(a,s)).
    \]
\end{itemize}
\end{definition}

Recall that $\overline{\varphi}$ computes a value for each variable according the auxiliary function. For an output port, the readout function $r(F, \varphi)$ sums the value of the variables to which that output port is linked. Recall that $F(fv)$ assigns to each flow a variable that defines its rate. So for a stock, the update function $u(F, \varphi)$ implies that the rate of change for that stock  is the sum of the rates of the inflows to that stock  minus the sum of the rates of the outflows from that stock. The following theorem implies that 
 the interpretation of stock-flow diagrams as Mealy machines respects composition.

\begin{restatable}[]{theorem}{alphanattrans}

    $\alpha \colon \SF\Rightarrow\Dynam$ is a monoidal natural transformation.
    
\end{restatable}

\begin{example}

    Let $(F,\varphi)$ be the stock-flow diagram for the SIR epidemic model defined in Example~\ref{ex:initial-sir-example}. 
    To compute $\ovarphi$ it suffices to compute the value of each variable in order of their topological sort. For example note that $p$ precedes $f$ which precedes $i$ in any topological sort of the variables. We can compute their values in order:
    \[
        p = \frac{I}{S + I + R}, \quad 
        f = c \beta \frac{I}{S + I + R}, \quad
        i = c \beta S \frac{I}{S + I + R}.
    \]
    
    $\alpha$ as defined in Definition~\ref{def:alpha} applied to this stock-flow diagram gives the Mealy machine consisting of:
    \begin{itemize}
        \item Three states corresponding to the susceptible, infected, and recovered population.
        \item Read out function $\R^4 \times \R^3 \to \R$ given by
        \[
            ((c, \beta, \tau , \omega), (S, I, R)) \mapsto \frac{c\beta SI}{S + I + R}.
        \]
        \item Update function given by 
        \[
            \dot S = \omega R - \frac{c\beta SI}{S + I + R}, \quad 
            \dot I  = \frac{c\beta SI}{S + I + R} - \tau I, \quad 
            \dot R  = \tau I - \omega R.
        \]
    \end{itemize}

\end{example}

\section{Conclusion}

In this paper, we developed a compositional pattern for composing input/output systems where the output can directly and instantaneously be affected by the input. We introduced the operad of dependent directed wiring diagrams and defined two algebras, $\Dynam$ and $\SF$, for composing Mealy machines and stock-flow diagrams, respectively. Finally, we defined a natural transformation $\SF \Rightarrow \Dynam$ which interprets a diagrammatic stock-flow diagram as a differential equation.

For future work, we are inspired to extend the work of~\cite{libkind2020rsm} to define an operad algebra that combines undirected approach to composing stock-flow diagrams presented in~\cite{sfcomp} with the directed approach presented in Section~\ref{section:SFsection}. We are further interested in implementations of the directed approach to composition in the Julia package StockFlow which was developed in~\cite{sfcomp} and can be found on GitHub at \url{https://github.com/AlgebraicJulia/StockFlow.jl}. Finally, we hope to develop these $\dDWD$ algebras into double categorical systems theories~\cite{jazmyers2021double}.

\nocite{*}
\bibliographystyle{eptcs}
\bibliography{generic}


\end{document}